\documentclass{amsproc}
\usepackage{amsmath}
\usepackage{amsfonts}

\theoremstyle{plain}

\newtheorem{example}{Example}

\newtheorem{lemma}{Lemma}

\newtheorem{theorem}{Theorem}
\numberwithin{equation}{section}
\input{tcilatex}

\begin{document}
\title[Bour's surfaces in $\mathbb{L}^{3}$]{Bour's spacelike maximal and
timelike minimal surfaces in the 3-dimensional Minkowski space }
\author{Erhan G\"{u}ler}
\address{Bart\i n University, Faculty of Science, Department of Mathematics,
Bart\i n, Turkey\\
}
\email[E. G\"{u}ler]{ergler@gmail.com}
\date{}
\dedicatory{}
\thanks{This paper is in final form and no version of it will be submitted
for publication elsewhere}

\begin{abstract}
Bour's\ minimal surface has remarkable properties in three dimensional
Minkowski space. We reveal the definite and indefinite cases of the Bour's
surface using Weierstrass representations, and give some differential
geometric properties of the astonishing maximal and minimal surfaces.
\end{abstract}

\maketitle

\section{Introduction}

The origins of minimal surface theory can be traced back to 1744 with the
Swedish mathematician Leonhard Euler's paper, and to the 1760 French
mathematician Joseph Louis Lagrange's paper.

A \textit{minimal surface} in three dimensional Euclidean space $\mathbb{E}%
^{3}$ is a regular surface for which the mean curvature vanishes
identically, firstly defined by Lagrange in 1760.

Classical minimal surfaces in Riemannian geometry are known by nearly all
the mathematicians, especially the geometers. However, there is a little
knowledge about the Bour's minimal surface. In 1862, the French
mathematician Jacques Edmond \'{E}mile Bour used semigeodesic coordinates
and found a number of new cases of deformations of surfaces. He gave a well
known theorem about the helicoidal and rotational surfaces. And also the
Bour-Enneper equation\ (today called the sine-Gordon wave equation)\ used in
soliton theory and\ quantum field theories in Physics was first set down by
Bour.

Minimal surfaces applicable onto a rotational surface were first determined
by Bour \cite{Bour} in 1862. These surfaces have been called $\ \mathfrak{B}%
_{m}$ (following Haag)\ to emphasize the value of $m$. We see the papers
dealing with the $\mathfrak{B}_{m}$ in the literature:

Bour, E. Theorie de la deformation des surfaces. Journal de l'\^{E}cole
Imperiale Polytechnique, tome 22, cahier 39 (1862), pp. 99-109.

Schwarz, H. A. Miscellen aus dem Gebiete der Minimalfl\"{a}chen. Journal de
Crelle, vol. 80 (1875), p. 295, published also in Gesammelte Mathematische
Abhandlungen.

Ribaucour, A. Etude sur les elassoides ou surfaces a courbure moyenne nulle.
Memoires Couronnes de l'Academie Royale de Belgique, vol. XLIV (1882),
chapter XX, pp. 215-224.

Demoulin, A. Bulletin des Sciences Mathematiques (2), vol. XXI (1897), pp.
244-252.

Haag, J. Bulletin des Sciences Mathematiques (2), vol. XXX (1906), pp.
75-94, also pp. 293-296.

St\"{u}bler, E. Mathematische Annalen, vol. 75 (1914), pp. 148-176.

Whittemore, J. K. Minimal surfaces applicable to surfaces of revolution.
Ann. of Math. (2) 19 (1917), no. 1, 1--20.

All real minimal surfaces applicable to rotational surfaces setting%
\begin{equation*}
F\mathfrak{(}s)=Cs^{m-2}
\end{equation*}%
in the Weierstrass representation equations, where $s,C\in \mathbb{C}$, $%
m\in \mathbb{R}$, and $F\mathfrak{(}s)$ is an analytic function. For $C=1,$ $%
m=0$ we obtain the Catenoid, $C=i,$ $m=0,$ the right Helicoid, $C=1,$ $m=2,$
the Enneper's surface (see also \cite{Gray, Nitsche}).

The Bour's surface has not been studied up till now in three dimensional
Minkowski space $\mathbb{L}^{3}$. In this paper, we reveal the Bour's
surface in $\mathbb{L}^{3}.$

Next, we focus on the definite (resp. indefinite) case of the Bour's maximal
(resp. minimal) surface in three dimensional Minkowski space $\mathbb{L}^{3}$%
.

\section{Bour's spacelike maximal surface $\mathfrak{B}_{m}$}

Throughout this work, we shall identify a vector $\overrightarrow{\mathbf{v}}%
=\left( u,v,w\right) $ with its transpose, the surfaces will be smooth, and
simply connected.

\subsection{Definite case}

Let $\mathbb{L}^{3}$ be a $3$-dimensional Minkowski space with natural
Lorentzian metric%
\begin{equation*}
\left\langle \text{ }.\text{ },\text{ }.\text{ }\right\rangle
=dx^{2}+dy^{2}-dz^{2}.
\end{equation*}%
A vector $\overrightarrow{\mathbf{v}}$ in $\mathbb{L}^{3}$ is called
spacelike if $\left\langle \overrightarrow{\mathbf{v}},\overrightarrow{%
\mathbf{v}}\right\rangle >0$ or $\overrightarrow{\mathbf{v}}=0,$ timelike if 
$\left\langle \overrightarrow{\mathbf{v}},\overrightarrow{\mathbf{v}}%
\right\rangle <0,$ lightlike if $\overrightarrow{\mathbf{v}}\neq 0$
satisfies $\left\langle \overrightarrow{\mathbf{v}},\overrightarrow{\mathbf{v%
}}\right\rangle =0.$

A surface in $\mathbb{L}^{3}$ is called a spacelike (resp. timelike,
degenere (lightlike)) if the induced metric on the surface is a positive
definite Riemannian (resp. Lorentzian, degenere) metric. A spacelike surface
with vanishing mean curvature is called a \textit{maximal surface} in three
dimensional Minkowski space.

The Weierstrass representation theorem for minimal surfaces in $\mathbb{E}%
^{3}$ is discovered by K. Weierstrass in 1866 \cite{Weierstrass}. Next, we
give it for the maximal surfaces in $\mathbb{L}^{3}.$

\begin{theorem}
(\textit{Weierstrass }representation for maximal surfaces in $\mathbb{L}^{3}$%
\textit{). }Let $\mathfrak{F}$ and $\mathcal{G}$ be two holomorphic
functions defined on a simply connected open subset $U$ of $\mathbb{C}$ such
that $\mathfrak{F}$ does not vanish and $\left\vert \mathcal{G}\right\vert
\neq 1$ on $U$. Then the map%
\begin{equation*}
\mathbf{x}\left( \zeta \right) =\func{Re}\int^{\zeta }\left( 
\begin{array}{c}
\mathfrak{F}\left( 1+\mathcal{G}^{2}\right) \\ 
i\text{ }\mathfrak{F}\left( 1-\mathcal{G}^{2}\right) \\ 
2\mathfrak{F}\mathcal{G}%
\end{array}%
\right) d\zeta
\end{equation*}%
is a conformal immersion of $U$ into $\mathbb{L}^{3}$ whose image is a
maximal surface.
\end{theorem}

See also details for the maximal surfaces in literature \cite{Anciaux,
Kobayashi, McNertney}.

\begin{lemma}
The Weierstrass patch determined by the functions%
\begin{equation*}
\left( \mathfrak{F}\left( \zeta \right) ,\mathcal{G}\left( \zeta \right)
\right) =\left( \zeta ^{m-2},\zeta \right)
\end{equation*}%
is a representation of the Bour's surface of value $\zeta \in 
\mathbb{C}
,$ $m\in \mathbb{R}$ in $\mathbb{L}^{3}.$
\end{lemma}

\begin{theorem}
Bour's surface of value $m$%
\begin{equation}
\mathfrak{B}_{m}\left( r,\theta \right) =\left( 
\begin{array}{c}
\frac{r^{m-1}}{m-1}\cos \left[ \left( m-1\right) \theta \right] +\frac{%
r^{m+1}}{m+1}\cos \left[ \left( m+1\right) \theta \right] \\ 
-\frac{r^{m-1}}{m-1}\sin \left[ \left( m-1\right) \theta \right] +\frac{%
r^{m+1}}{m+1}\sin \left[ \left( m+1\right) \theta \right] \\ 
2\frac{r^{m}}{m}\cos \left( m\theta \right)%
\end{array}%
\right)  \label{7}
\end{equation}%
is a maximal surface in $\mathbb{L}^{3}$, where $m\in \mathbb{R}-\left\{
-1,0,1\right\} $.
\end{theorem}

\begin{proof}
The\textbf{\ }coefficients\textbf{\ }of the\textbf{\ }first fundamental form
of the surface $\mathfrak{B}_{m}$ are%
\begin{eqnarray*}
E &=&r^{2m-4}\left( 1-r^{2}\right) ^{2}, \\
F &=&0, \\
G &=&r^{2m-2}\left( 1-r^{2}\right) ^{2}.
\end{eqnarray*}%
We have%
\begin{equation*}
\det I=\left[ r^{2m-3}\left( 1-r^{2}\right) ^{2}\right] ^{2}.
\end{equation*}%
So, $\mathfrak{B}_{m}$ is a spacelike surface. The Gauss map of the surface
is%
\begin{equation*}
e=\frac{1}{r^{2}-1}\left( 
\begin{array}{c}
2r\cos (\theta ) \\ 
2r\sin (\theta ) \\ 
r^{2}+1%
\end{array}%
\right) .
\end{equation*}%
The\textbf{\ }coefficients of the\ second fundamental form of the surface are%
\begin{eqnarray*}
L &=&2r^{m-2}\cos \left( m\theta \right) , \\
M &=&-2r^{m-1}\sin \left( m\theta \right) , \\
N &=&-2r^{m}\cos \left( m\theta \right) .
\end{eqnarray*}%
Then, we have%
\begin{equation*}
\det II=-4r^{2m-2}.
\end{equation*}%
In spacelike case, the Gaussian curvature is defined by%
\begin{equation*}
K=\epsilon \text{ }\frac{\det II}{\left\vert \det I\right\vert },
\end{equation*}%
where $\epsilon :=\left\langle e,e\right\rangle =-1$ in $\mathbb{L}^{3}.$
Hence, the Gaussian curvature and the mean curvature of the Bour's surface
of value $m$, respectively, are%
\begin{equation*}
K=\left( \frac{2r^{2-m}}{\left( 1-r^{2}\right) ^{2}}\right) ^{2},\text{ \ }%
H=0.
\end{equation*}%
Therefore, the $\mathfrak{B}_{m}$ is a maximal surface in $\mathbb{L}^{3}.$
\end{proof}

\subsection{Bour's spacelike surface $\mathfrak{B}_{3}$}

If take $m=3$ in $\mathfrak{B}_{m}\left( r,\theta \right) $, we have \textit{%
Bour's maximal surface}\textbf{\ }$\mathfrak{B}_{3}$\textbf{\ }(see Fig. 1
and Fig. 2)%
\begin{equation}
\mathfrak{B}_{3}\left( r,\theta \right) =\left( 
\begin{array}{c}
\frac{r^{2}}{2}\cos \left( 2\theta \right) +\frac{r^{4}}{4}\cos \left(
4\theta \right) \\ 
-\frac{r^{2}}{2}\sin \left( 2\theta \right) +\frac{r^{4}}{4}\sin \left(
4\theta \right) \\ 
\frac{2}{3}r^{3}\cos \left( 3\theta \right)%
\end{array}%
\right)  \label{8}
\end{equation}%
in Minkowski 3-space, where $r\in \lbrack -1,1]$, $\theta \in \lbrack 0,\pi
].$

\begin{center}
\begin{equation*}
\FRAME{itbpF}{1.8671in}{1.7435in}{0in}{}{}{Figure}{\special{language
"Scientific Word";type "GRAPHIC";maintain-aspect-ratio TRUE;display
"USEDEF";valid_file "T";width 1.8671in;height 1.7435in;depth
0in;original-width 3.0934in;original-height 2.885in;cropleft "0";croptop
"1";cropright "1";cropbottom "0";tempfilename
'MG5RC30U.wmf';tempfile-properties "XPR";}}\FRAME{itbpF}{1.8983in}{1.7513in}{%
0in}{}{}{Figure}{\special{language "Scientific Word";type
"GRAPHIC";maintain-aspect-ratio TRUE;display "USEDEF";valid_file "T";width
1.8983in;height 1.7513in;depth 0in;original-width 3.0519in;original-height
2.8124in;cropleft "0";croptop "1";cropright "1";cropbottom "0";tempfilename
'MG5RC30V.wmf';tempfile-properties "XPR";}}
\end{equation*}%
$\ \ \ \ \left( a\right) $ \ \ \ \ \ \ \ \ \ \ \ \ \ \ \ \ \ \ \ \ \ \ \ \ \
\ \ \ \ \ \ \ $\left( b\right) $ \ \ \ \ 

Figure 1 \ Bour's maximal surface $\mathfrak{B}_{3}\left( r,\theta \right) $

\begin{equation*}
\FRAME{itbpF}{3.4359in}{3.4489in}{0in}{}{}{Figure}{\special{language
"Scientific Word";type "GRAPHIC";maintain-aspect-ratio TRUE;display
"USEDEF";valid_file "T";width 3.4359in;height 3.4489in;depth
0in;original-width 5.6879in;original-height 5.7086in;cropleft "0";croptop
"1";cropright "1";cropbottom "0";tempfilename
'MHGCK700.wmf';tempfile-properties "XPR";}}
\end{equation*}

Figure 2 \ Maximal $\mathfrak{B}_{3}\left( r,\theta \right) $ with its
shadows
\end{center}

The\textbf{\ }coefficients\textbf{\ }of the\textbf{\ }first fundamental form
of the Bour's maximal surface of value $3$ are%
\begin{equation*}
E=r^{2}\left( 1-r^{2}\right) ^{2},\text{ }F=0,\text{ }G=r^{4}\left(
1-r^{2}\right) ^{2}.
\end{equation*}%
So,%
\begin{equation*}
\det I=r^{6}\left( 1-r^{2}\right) ^{4}.
\end{equation*}%
The Gauss map of the surface is%
\begin{equation*}
e=\frac{1}{r^{2}-1}\left( 2r\cos (\theta ),2r\sin (\theta ),1+r^{2}\right) .
\end{equation*}%
The\textbf{\ }coefficients of the\ second fundamental form of the surface are%
\begin{equation*}
L=2r\cos \left( 3\theta \right) ,\text{ }M=-2r^{2}\sin \left( 3\theta
\right) ,\text{ }N=-2r^{3}\cos \left( 3\theta \right) .
\end{equation*}%
Then,%
\begin{equation*}
\det II=-4r^{4}.
\end{equation*}%
The mean and the Gaussian curvatures of the Bour's maximal surface of value $%
3$ are, respectively,%
\begin{equation*}
H=0,\text{ \ }K=\frac{4}{r^{2}\left( 1-r^{2}\right) ^{4}}.
\end{equation*}%
The Weierstrass patch determined by the functions%
\begin{equation*}
\left( \mathfrak{F},\mathcal{G}\right) =\left( \zeta ,\zeta \right)
\end{equation*}%
is a representation of the Bour's maximal surface of value $3$ in $\mathbb{L}%
^{3}.$ The parametric form of the maximal surface $\mathfrak{B}_{3}$ (see
Fig. 3, and Fig. 4) is%
\begin{equation}
\mathfrak{B}_{3}\left( u,v\right) =\left( 
\begin{array}{c}
\frac{u^{4}}{4}+\frac{v^{4}}{4}-\frac{3}{2}u^{2}v^{2}+\frac{u^{2}}{2}-\frac{%
v^{2}}{2} \\ 
u^{3}v-uv^{3}-uv \\ 
\frac{2}{3}u^{3}-2uv^{2}%
\end{array}%
\right) ,  \label{9}
\end{equation}%
where $u,v\in \mathbb{R}$.

\begin{equation*}
\FRAME{itbpF}{2.175in}{1.8671in}{0in}{}{}{Figure}{\special{language
"Scientific Word";type "GRAPHIC";maintain-aspect-ratio TRUE;display
"USEDEF";valid_file "T";width 2.175in;height 1.8671in;depth
0in;original-width 4.0836in;original-height 3.4999in;cropleft "0";croptop
"1";cropright "1";cropbottom "0";tempfilename
'MG5RC30W.wmf';tempfile-properties "XPR";}}\FRAME{itbpF}{2.1326in}{1.8741in}{%
0in}{}{}{Figure}{\special{language "Scientific Word";type
"GRAPHIC";maintain-aspect-ratio TRUE;display "USEDEF";valid_file "T";width
2.1326in;height 1.8741in;depth 0in;original-width 3.9271in;original-height
3.448in;cropleft "0";croptop "1";cropright "1";cropbottom "0";tempfilename
'MG5RC30X.wmf';tempfile-properties "XPR";}}
\end{equation*}

\begin{center}
$\ \left( a\right) $ \ \ \ \ \ \ \ \ \ \ \ \ \ \ \ \ \ \ \ \ \ \ \ \ \ \ \ \
\ \ \ \ \ \ \ \ \ $\left( b\right) $

Figure 3 \ Maximal surface $\mathfrak{B}_{3}\left( u,v\right) ,$ $u,v\in
\lbrack -1,1]$

\begin{equation*}
\FRAME{itbpF}{3.2975in}{4.0023in}{0in}{}{}{Figure}{\special{language
"Scientific Word";type "GRAPHIC";maintain-aspect-ratio TRUE;display
"USEDEF";valid_file "T";width 3.2975in;height 4.0023in;depth
0in;original-width 4.7816in;original-height 5.8124in;cropleft "0";croptop
"1";cropright "1";cropbottom "0";tempfilename
'MHGCPB01.wmf';tempfile-properties "XPR";}}
\end{equation*}

Figure 4 \ Maximal surface $\mathfrak{B}_{3}\left( u,v\right) $ with its
shadows
\end{center}

The\textbf{\ }coefficients\textbf{\ }of the\textbf{\ }first fundamental form
of the Bour's surface of value $3$ are%
\begin{equation*}
E=\left( u^{2}+v^{2}\right) \left( u^{2}+v^{2}-1\right) ^{2}=G,\text{ \ }F=0,
\end{equation*}%
So,%
\begin{equation*}
\det I=\left( u^{2}+v^{2}\right) ^{2}\left( u^{2}+v^{2}-1\right) ^{4}.
\end{equation*}%
The Gauss map of the surface $\mathfrak{B}_{3}$ is%
\begin{equation*}
e=\frac{1}{u^{2}+v^{2}-1}\left( 2u,2v,u^{2}+v^{2}+1\right) .
\end{equation*}%
The\textbf{\ }coefficients of the\ second fundamental form of the surface are%
\begin{equation*}
L=2u,\text{ }M=-2v,\text{ }N=-2u.
\end{equation*}%
Then,%
\begin{equation*}
\det II=-4\left( u^{2}+v^{2}\right) .
\end{equation*}%
The mean and the Gaussian curvatures of the Bour's minimal surface of value $%
3$ are%
\begin{equation*}
H=0,\text{ \ }K=\frac{4}{\left( u^{2}+v^{2}\right) \left(
1+u^{2}+v^{2}\right) ^{4}}.
\end{equation*}

\subsection{Applications of the definite case in $\mathbb{L}^{3}$}

\begin{example}
If take $m=2$, we have\textbf{\ }Enneper's maximal surface\textbf{\ }without
self-intersections\textbf{\ }(see Fig. 5)%
\begin{equation*}
\mathfrak{B}_{2}\left( r,\theta \right) =\left( 
\begin{array}{c}
r\cos \left( \theta \right) +\frac{r^{3}}{3}\cos \left( 3\theta \right) \\ 
-r\sin \left( \theta \right) +\frac{r^{3}}{3}\sin \left( 3\theta \right) \\ 
r^{2}\cos \left( 2\theta \right)%
\end{array}%
\right) ,
\end{equation*}%
and $\left( \mathfrak{F},\mathcal{G}\right) =\left( 1,\zeta \right) $ in
Minkowski 3-space, where $r\in \lbrack -1,1]$, $\theta \in \lbrack 0,\pi ].$
\end{example}

\begin{center}
\begin{equation*}
\FRAME{itbpF}{1.7435in}{1.8066in}{0in}{}{}{Figure}{\special{language
"Scientific Word";type "GRAPHIC";display "USEDEF";valid_file "T";width
1.7435in;height 1.8066in;depth 0in;original-width 3.1773in;original-height
3.6461in;cropleft "0";croptop "1";cropright "1";cropbottom "0";tempfilename
'MG5RC30Y.wmf';tempfile-properties "XPR";}}\FRAME{itbpF}{1.9199in}{1.8005in}{%
0in}{}{}{Figure}{\special{language "Scientific Word";type
"GRAPHIC";maintain-aspect-ratio TRUE;display "USEDEF";valid_file "T";width
1.9199in;height 1.8005in;depth 0in;original-width 3.3018in;original-height
3.0934in;cropleft "0";croptop "1";cropright "1";cropbottom "0";tempfilename
'MG5RC30Z.wmf';tempfile-properties "XPR";}}
\end{equation*}%
$\left( a\right) $ \ \ \ \ \ \ \ \ \ \ \ \ \ \ \ \ \ \ \ \ \ \ \ \ \ \ \ \ \
\ \ $\left( b\right) $

Figure 5 \ Maximal surface $\mathfrak{B}_{2}$ without self-intersections
\end{center}

\begin{example}
If take $m=2$, we have Enneper's maximal surface\textbf{\ }with
self-intersections\textbf{\ }(see Fig. 6) in Minkowski 3-space, where $r\in
\lbrack -3,3]$, $\theta \in \lbrack 0,\pi ].$
\end{example}

\begin{center}
\begin{equation*}
\FRAME{itbpF}{1.8887in}{1.9069in}{0in}{}{}{Figure}{\special{language
"Scientific Word";type "GRAPHIC";display "USEDEF";valid_file "T";width
1.8887in;height 1.9069in;depth 0in;original-width 3.2603in;original-height
3.4065in;cropleft "0";croptop "1";cropright "1";cropbottom "0";tempfilename
'MG5RC310.wmf';tempfile-properties "XPR";}}\FRAME{itbpF}{2.1093in}{1.8948in}{%
0in}{}{}{Figure}{\special{language "Scientific Word";type
"GRAPHIC";maintain-aspect-ratio TRUE;display "USEDEF";valid_file "T";width
2.1093in;height 1.8948in;depth 0in;original-width 3.5518in;original-height
3.1877in;cropleft "0";croptop "1";cropright "1";cropbottom "0";tempfilename
'MG5RC311.wmf';tempfile-properties "XPR";}}
\end{equation*}%
$\left( a\right) $ \ \ \ \ \ \ \ \ \ \ \ \ \ \ \ \ \ \ \ \ \ \ \ \ \ \ \ \ \
\ \ \ \ \ $\left( b\right) $

Figure 6 $\ $Maximal surface $\mathfrak{B}_{2}$ with self-intersections
\end{center}

\begin{example}
If take $m=\frac{1}{2}$, we have\textbf{\ }(see Fig. 7)%
\begin{equation*}
\mathfrak{B}_{1/2}\left( r,\theta \right) =\left( 
\begin{array}{c}
-2r^{-1/2}\cos \left( \frac{\theta }{2}\right) +\frac{2}{3}r^{3/2}\cos
\left( \frac{3\theta }{2}\right) \\ 
-2r^{-1/2}\sin \left( \frac{\theta }{2}\right) +\frac{2}{3}r^{3/2}\sin
\left( \frac{3\theta }{2}\right) \\ 
4r^{1/2}\cos \left( \frac{\theta }{2}\right)%
\end{array}%
\right) ,
\end{equation*}%
and $\left( \mathfrak{F},\mathcal{G}\right) =\left( \zeta ^{-3/2},\zeta
\right) $ in Minkowski 3-space, where $r\in \lbrack -1,1]$, $\theta \in
\lbrack -2\pi ,2\pi ].$
\end{example}

\begin{center}
\begin{equation*}
\FRAME{itbpF}{2.1283in}{1.881in}{0in}{}{}{Figure}{\special{language
"Scientific Word";type "GRAPHIC";maintain-aspect-ratio TRUE;display
"USEDEF";valid_file "T";width 2.1283in;height 1.881in;depth
0in;original-width 3.2811in;original-height 2.8963in;cropleft "0";croptop
"1";cropright "1";cropbottom "0";tempfilename
'MG5RC312.wmf';tempfile-properties "XPR";}}\FRAME{itbpF}{1.9925in}{1.8741in}{%
0in}{}{}{Figure}{\special{language "Scientific Word";type
"GRAPHIC";maintain-aspect-ratio TRUE;display "USEDEF";valid_file "T";width
1.9925in;height 1.8741in;depth 0in;original-width 3.1462in;original-height
2.9585in;cropleft "0";croptop "1";cropright "1";cropbottom "0";tempfilename
'MG5RC313.wmf';tempfile-properties "XPR";}}
\end{equation*}%
$\left( a\right) $ \ \ \ \ \ \ \ \ \ \ \ \ \ \ \ \ \ \ \ \ \ \ \ \ \ \ \ \ \
\ \ \ \ \ \ \ \ \ $\left( b\right) $

Figure 7 \ Maximal surface $\mathfrak{B}_{1/2}$
\end{center}

\begin{example}
If $m=\frac{3}{2}$, we have (see Fig. 8)%
\begin{equation*}
\mathfrak{B}_{3/2}\left( r,\theta \right) =\left( 
\begin{array}{c}
2r^{-1/2}\cos \left( \frac{\theta }{2}\right) +\frac{2}{5}r^{5/2}\cos \left( 
\frac{5\theta }{2}\right) \\ 
-2r^{-1/2}\sin \left( \frac{\theta }{2}\right) +\frac{2}{5}r^{5/2}\sin
\left( \frac{5\theta }{2}\right) \\ 
\frac{4}{3}r^{3/2}\cos \left( \frac{3\theta }{2}\right)%
\end{array}%
\right) ,
\end{equation*}%
with self-intersections, and $\left( \mathfrak{F},\mathcal{G}\right) =\left(
\zeta ^{-1/2},\zeta \right) $ in Minkowski 3-space, where $r\in \lbrack
-3,3] $, $\theta \in \lbrack -2\pi ,2\pi ].$
\end{example}

\begin{center}
\begin{equation*}
\FRAME{itbpF}{1.8299in}{1.7841in}{0in}{}{}{Figure}{\special{language
"Scientific Word";type "GRAPHIC";maintain-aspect-ratio TRUE;display
"USEDEF";valid_file "T";width 1.8299in;height 1.7841in;depth
0in;original-width 3.7395in;original-height 3.6461in;cropleft "0";croptop
"1";cropright "1";cropbottom "0";tempfilename
'MG5RC314.wmf';tempfile-properties "XPR";}}\FRAME{itbpF}{2.047in}{1.7685in}{%
0in}{}{}{Figure}{\special{language "Scientific Word";type
"GRAPHIC";maintain-aspect-ratio TRUE;display "USEDEF";valid_file "T";width
2.047in;height 1.7685in;depth 0in;original-width 3.5725in;original-height
3.0831in;cropleft "0";croptop "1";cropright "1";cropbottom "0";tempfilename
'MG5RC315.wmf';tempfile-properties "XPR";}}
\end{equation*}%
$\left( a\right) $ \ \ \ \ \ \ \ \ \ \ \ \ \ \ \ \ \ \ \ \ \ \ \ \ \ \ \ \ \
\ \ \ \ \ \ $\left( b\right) $ \ \ 

Figure 8 \ Maximal surface $\mathfrak{B}_{3/2}$
\end{center}

\begin{example}
If $m=\frac{3}{2}$, we have $\mathfrak{B}_{3/2}\left( r,\theta \right) $
without self-intersections (see Fig. 9), and $\left( \mathfrak{F},\mathcal{G}%
\right) =\left( z^{-1/2},z\right) $ in Minkowski 3-space, where $r\in
\lbrack -1,1]$, $\theta \in \lbrack -2\pi ,2\pi ].$
\end{example}

\begin{center}
\begin{equation*}
\FRAME{itbpF}{1.8922in}{1.7772in}{0in}{}{}{Figure}{\special{language
"Scientific Word";type "GRAPHIC";display "USEDEF";valid_file "T";width
1.8922in;height 1.7772in;depth 0in;original-width 3.8337in;original-height
3.4272in;cropleft "0";croptop "1";cropright "1";cropbottom "0";tempfilename
'MG5RC316.wmf';tempfile-properties "XPR";}}\text{\FRAME{itbpF}{1.9112in}{%
1.7737in}{0in}{}{}{Figure}{\special{language "Scientific Word";type
"GRAPHIC";display "USEDEF";valid_file "T";width 1.9112in;height
1.7737in;depth 0in;original-width 3.531in;original-height 3.4065in;cropleft
"0";croptop "1";cropright "1";cropbottom "0";tempfilename
'MG5RC317.wmf';tempfile-properties "XPR";}}}
\end{equation*}%
\ $\left( a\right) $ \ \ \ \ \ \ \ \ \ \ \ \ \ \ \ \ \ \ \ \ \ \ \ \ \ \ \ \
\ \ \ \ \ \ $\left( b\right) $ \ \ 

Figure 9 \ Maximal surface $\mathfrak{B}_{3/2}$
\end{center}

\begin{example}
If $m=\frac{2}{3}$, we have (see Fig. 10)%
\begin{equation*}
\mathfrak{B}_{2/3}\left( r,\theta \right) =\left( 
\begin{array}{c}
-3r^{-1/3}\cos \left( \frac{\theta }{3}\right) +\frac{3}{5}r^{5/3}\cos
\left( \frac{5\theta }{3}\right) \\ 
-3r^{-1/3}\sin \left( \frac{\theta }{3}\right) +\frac{3}{5}r^{5/3}\sin
\left( \frac{5\theta }{3}\right) \\ 
3r^{2/3}\cos \left( \frac{2\theta }{3}\right)%
\end{array}%
\right) ,
\end{equation*}%
and $\left( \mathfrak{F},\mathcal{G}\right) =\left( \zeta ^{-4/3},\zeta
\right) $ in Minkowski 3-space, where $r\in \lbrack -1,1]$, $\theta \in
\lbrack -3\pi ,3\pi ].$
\end{example}

\begin{center}
\begin{equation*}
\FRAME{itbpF}{1.9735in}{1.8239in}{0in}{}{}{Figure}{\special{language
"Scientific Word";type "GRAPHIC";maintain-aspect-ratio TRUE;display
"USEDEF";valid_file "T";width 1.9735in;height 1.8239in;depth
0in;original-width 3.4065in;original-height 3.1462in;cropleft "0";croptop
"1";cropright "1";cropbottom "0";tempfilename
'MG5RC318.wmf';tempfile-properties "XPR";}}\FRAME{itbpF}{2.0116in}{1.8178in}{%
0in}{}{}{Figure}{\special{language "Scientific Word";type
"GRAPHIC";maintain-aspect-ratio TRUE;display "USEDEF";valid_file "T";width
2.0116in;height 1.8178in;depth 0in;original-width 4.0871in;original-height
3.691in;cropleft "0";croptop "1";cropright "1";cropbottom "0";tempfilename
'MG5RC319.bmp';tempfile-properties "XPR";}}
\end{equation*}%
$\left( a\right) $ \ \ \ \ \ \ \ \ \ \ \ \ \ \ \ \ \ \ \ \ \ \ \ \ \ \ \ \ \
\ \ \ \ \ \ $\left( b\right) $

Figure 10 \ Maximal surface $\mathfrak{B}_{2/3}$
\end{center}

\begin{example}
If $m=\frac{4}{3}$, then we have (see Fig. 11)%
\begin{equation*}
\mathfrak{B}_{4/3}\left( r,\theta \right) =\left( 
\begin{array}{c}
3r^{1/3}\cos \left( \frac{\theta }{3}\right) +\frac{3}{7}r^{7/3}\cos \left( 
\frac{7\theta }{3}\right) \\ 
-3r^{1/3}\sin \left( \frac{\theta }{3}\right) +\frac{3}{7}r^{7/3}\sin \left( 
\frac{7\theta }{3}\right) \\ 
\frac{3}{2}r^{4/3}\cos \left( \frac{4\theta }{3}\right)%
\end{array}%
\right) ,
\end{equation*}%
and $\left( \mathfrak{F},\mathcal{G}\right) =\left( \zeta ^{-2/3},\zeta
\right) $ in Minkowski 3-space, where $r\in \lbrack -2,2]$, $\theta \in
\lbrack -3\pi ,3\pi ].$
\end{example}

\begin{center}
\begin{equation*}
\FRAME{itbpF}{1.7876in}{1.8524in}{0in}{}{}{Figure}{\special{language
"Scientific Word";type "GRAPHIC";maintain-aspect-ratio TRUE;display
"USEDEF";valid_file "T";width 1.7876in;height 1.8524in;depth
0in;original-width 3.6876in;original-height 3.8233in;cropleft "0";croptop
"1";cropright "1";cropbottom "0";tempfilename
'MG5RC31A.wmf';tempfile-properties "XPR";}}\FRAME{itbpF}{2.047in}{1.8421in}{%
0in}{}{}{Figure}{\special{language "Scientific Word";type
"GRAPHIC";maintain-aspect-ratio TRUE;display "USEDEF";valid_file "T";width
2.047in;height 1.8421in;depth 0in;original-width 3.5103in;original-height
3.1566in;cropleft "0";croptop "1";cropright "1";cropbottom "0";tempfilename
'MG5RC31B.wmf';tempfile-properties "XPR";}}
\end{equation*}%
$\left( a\right) $ \ \ \ \ \ \ \ \ \ \ \ \ \ \ \ \ \ \ \ \ \ \ \ \ \ \ \ \ \
\ \ \ $\left( b\right) $ \ 

Figure 11 \ Maximal surface $\mathfrak{B}_{4/3}$
\end{center}

\begin{example}
If $m=\frac{5}{2}$, then we have (see Fig. 12)%
\begin{equation*}
\mathfrak{B}_{5/2}\left( r,\theta \right) =\left( 
\begin{array}{c}
\frac{2}{3}r^{3/2}\cos \left( \frac{3\theta }{2}\right) +\frac{2}{7}%
r^{7/2}\cos \left( \frac{7\theta }{2}\right) \\ 
-\frac{2}{3}r^{3/2}\sin \left( \frac{3\theta }{2}\right) +\frac{2}{7}%
r^{7/2}\sin \left( \frac{7\theta }{2}\right) \\ 
\frac{4}{5}r^{5/2}\cos \left( \frac{5\theta }{2}\right)%
\end{array}%
\right) ,
\end{equation*}%
and $\left( \mathfrak{F},\mathcal{G}\right) =\left( \zeta ^{1/2},\zeta
\right) $ in Minkowski 3-space, where $r\in \lbrack -1,1]$, $\theta \in
\lbrack -2\pi ,2\pi ].$
\end{example}

\begin{center}
\begin{equation*}
\FRAME{itbpF}{1.7781in}{1.7962in}{0in}{}{}{Figure}{\special{language
"Scientific Word";type "GRAPHIC";display "USEDEF";valid_file "T";width
1.7781in;height 1.7962in;depth 0in;original-width 3.8121in;original-height
4.094in;cropleft "0";croptop "1";cropright "1";cropbottom "0";tempfilename
'MG5RC31C.wmf';tempfile-properties "XPR";}}\FRAME{itbpF}{2.0159in}{1.7625in}{%
0in}{}{}{Figure}{\special{language "Scientific Word";type
"GRAPHIC";maintain-aspect-ratio TRUE;display "USEDEF";valid_file "T";width
2.0159in;height 1.7625in;depth 0in;original-width 3.531in;original-height
3.0831in;cropleft "0";croptop "1";cropright "1";cropbottom "0";tempfilename
'MG5RC31D.wmf';tempfile-properties "XPR";}}
\end{equation*}%
$\left( a\right) $ \ \ \ \ \ \ \ \ \ \ \ \ \ \ \ \ \ \ \ \ \ \ \ \ \ \ \ \ \
\ \ $\left( b\right) $ \ 

Figure 12 $\ $Maximal surface $\mathfrak{B}_{5/2}$
\end{center}

\begin{example}
If $m=4$, then we have (see Fig. 13)%
\begin{equation*}
\mathfrak{B}_{4}\left( r,\theta \right) =\left( 
\begin{array}{c}
\frac{1}{3}r^{3}\cos \left( 3\theta \right) +\frac{1}{5}r^{5}\cos \left(
5\theta \right) \\ 
-\frac{1}{3}r^{3}\sin \left( 3\theta \right) +\frac{1}{5}r^{5}\sin \left(
5\theta \right) \\ 
\frac{1}{2}r^{4}\cos \left( 4\theta \right)%
\end{array}%
\right) ,
\end{equation*}%
$\allowbreak $ and $\left( \mathfrak{F},\mathcal{G}\right) =\left( \zeta
^{2},\zeta \right) $ in Minkowski 3-space, where $r\in \lbrack -1,1]$, $%
\theta \in \lbrack 0,2\pi ].$
\end{example}

\begin{center}
\begin{equation*}
\FRAME{itbpF}{1.7426in}{1.8005in}{0in}{}{}{Figure}{\special{language
"Scientific Word";type "GRAPHIC";maintain-aspect-ratio TRUE;display
"USEDEF";valid_file "T";width 1.7426in;height 1.8005in;depth
0in;original-width 3.0934in;original-height 3.1981in;cropleft "0";croptop
"1";cropright "1";cropbottom "0";tempfilename
'MG5RC31E.wmf';tempfile-properties "XPR";}}\FRAME{itbpF}{2.0557in}{1.7365in}{%
0in}{}{}{Figure}{\special{language "Scientific Word";type
"GRAPHIC";maintain-aspect-ratio TRUE;display "USEDEF";valid_file "T";width
2.0557in;height 1.7365in;depth 0in;original-width 3.1254in;original-height
2.6351in;cropleft "0";croptop "1";cropright "1";cropbottom "0";tempfilename
'MG5RC31F.wmf';tempfile-properties "XPR";}}
\end{equation*}%
$\left( a\right) $ \ \ \ \ \ \ \ \ \ \ \ \ \ \ \ \ \ \ \ \ \ \ \ \ \ \ \ \ \
\ \ $\left( b\right) $ \ 

Figure 13 $\ $Maximal surface $\mathfrak{B}_{4}$
\end{center}

\section{Bour's timelike minimal surface $\mathfrak{B}_{m}$}

\subsection{Indefinite case}

Let $\mathbb{L}^{2}=\left( \mathbb{R}^{2},-dx^{2}+dy^{2}\right) $ be
Minkowski plane, and $\mathbb{L}^{3}$ be a $3$-dimensional Minkowski space
with natural Lorentzian metric%
\begin{equation*}
\left\langle \text{ }.\text{ },\text{ }.\text{ }\right\rangle
=-dx^{2}+dy^{2}+dz^{2}.
\end{equation*}

\begin{theorem}
(\textit{Weierstrass }representation for timelike minimal surfaces in $%
\mathbb{L}^{3}$\textit{) }Let $\mathbf{x}:\mathbf{M}\rightarrow \mathbb{L}%
^{3}$ be a timelike surface parametrized by null coordinates $(u,v)$, where $%
u:=-x+y$, $v:=x+y$. Timelike minimal surface is represented by%
\begin{equation}
\mathbf{x}(u,v)=\int^{u}\left( 
\begin{array}{c}
-\emph{f}\left( 1+\emph{g}^{2}\right) \\ 
\emph{f}\left( 1-\emph{g}^{2}\right) \\ 
2\emph{fg}%
\end{array}%
\right) du+\int^{v}\left( 
\begin{array}{c}
\mathfrak{f}\left( 1+\mathfrak{g}^{2}\right) \\ 
\mathfrak{f}\left( 1-\mathfrak{g}^{2}\right) \\ 
2\mathfrak{fg}%
\end{array}%
\right) dv.  \label{10}
\end{equation}
\end{theorem}

The functions $\emph{f}(u)\emph{,}$ $\emph{g}(u)$, $\mathfrak{f}(v)$ and $%
\mathfrak{g}(v)$ are defined by%
\begin{equation*}
\emph{f}=\frac{-\phi _{1}+\phi _{2}}{2},\text{ }\emph{g}=\frac{\phi _{3}}{%
-\phi _{1}+\phi _{2}},
\end{equation*}%
\begin{equation*}
\mathfrak{f}=\frac{\mu _{1}+\mu _{2}}{2},\text{ }\mathfrak{g}=\frac{\mu _{3}%
}{\mu _{1}+\mu _{2}},
\end{equation*}%
and $\phi =\left( \phi _{1},\phi _{2},\phi _{3}\right) ,$ $\mu =\left( \mu
_{1},\mu _{2},\mu _{3}\right) $ vector valued functions, $\phi \left(
u\right) :=\mathbf{x}_{u},$ $\mu \left( v\right) :=\mathbf{x}_{v}$ satisfy%
\begin{equation*}
\left( \phi \right) ^{2}=0,\text{ }\left( \mu \right) ^{2}=0.
\end{equation*}%
Hence, the timelike minimal surface has the form%
\begin{eqnarray*}
\mathbf{x}(u,v) &=&\int^{u}\phi \left( u\right) du+\int^{v}\mu \left(
v\right) dv \\
&=&\Omega \left( u\right) +\Psi \left( v\right) ,
\end{eqnarray*}%
and its conjugate%
\begin{equation*}
\mathbf{x}^{\ast }(u,v)=\Omega \left( u\right) -\Psi \left( v\right) ,
\end{equation*}%
where $\phi \left( u\right) $ and $\mu \left( v\right) $ are linearly
independent, $\Omega \left( u\right) $ and $\Psi \left( v\right) $ are null
curves in $\mathbb{L}^{3}$. Weierstrass formula for the timelike minimal
surfaces obtained by M. Magid \cite{Magid} (see also \cite{Inoguchi and Lee}%
, for details).

\begin{lemma}
The Weierstrass patch determined by the functions%
\begin{equation*}
\left( \emph{f}(u)\emph{,g}(u)\right) =\left( u^{m-2},u\right) \text{ \ and
\ }\left( \mathfrak{f}(v),\mathfrak{g}(v)\right) =\left( v^{m-2},v\right)
\end{equation*}%
is a representation of the Bour's timelike minimal surface of value $m$ in $%
\mathbb{L}^{3}$, where $m\in \mathbb{R}.$
\end{lemma}

Bour's timelike minimal surface of value $m$ is%
\begin{equation*}
\int^{u}\left( 
\begin{array}{c}
-u^{m-2}\left( 1+u^{2}\right) \\ 
u^{m-2}\left( 1-u^{2}\right) \\ 
2u^{m-1}%
\end{array}%
\right) du+\int^{v}\left( 
\begin{array}{c}
v^{m-2}\left( 1+v^{2}\right) \\ 
v^{m-2}\left( 1-v^{2}\right) \\ 
2v^{m-1}%
\end{array}%
\right) dv,
\end{equation*}%
and it has the form%
\begin{equation}
\mathfrak{B}_{m}(u,v)=\left( 
\begin{array}{c}
-\frac{1}{m-1}\left( u^{m-1}-v^{m-1}\right) -\frac{1}{m+1}\left(
u^{m+1}-v^{m+1}\right) \\ 
\frac{1}{m-1}\left( u^{m-1}+v^{m-1}\right) -\frac{1}{m+1}\left(
u^{m+1}+v^{m+1}\right) \\ 
2\frac{1}{m}\left( u^{m}+v^{m}\right)%
\end{array}%
\right) .  \label{11}
\end{equation}%
Therefore, $\mathfrak{B}_{m}\left( r,\theta \right) $ is%
\begin{eqnarray*}
x &=&-\frac{r^{m-1}}{m-1}\left( \cos ^{\left( m-1\right) }\left( \theta
\right) -\sin ^{\left( m-1\right) }\left( \theta \right) \right) \\
&&-\frac{r^{m+1}}{m+1}\left( \cos ^{\left( m+1\right) }\left( \theta \right)
-\sin ^{\left( m+1\right) }\left( \theta \right) \right) ,
\end{eqnarray*}%
\begin{eqnarray*}
y &=&\frac{r^{m-1}}{m-1}\left( \cos ^{\left( m-1\right) }\left( \theta
\right) +\sin ^{\left( m-1\right) }\left( \theta \right) \right) \\
&&-\frac{r^{m+1}}{m+1}\left( \cos ^{\left( m+1\right) }\left( \theta \right)
+\sin ^{\left( m+1\right) }\left( \theta \right) \right) ,
\end{eqnarray*}%
\begin{equation*}
z\text{ }=\text{ }2\frac{r^{m}}{m}\left( \cos ^{m}\left( \theta \right)
+\sin ^{m}\left( \theta \right) \right) .\text{ \ \ \ \ \ \ \ \ \ \ \ \ \ \
\ }
\end{equation*}

\begin{theorem}
Bour's surface $\mathfrak{B}_{m}\left( r,\theta \right) $ is a timelike
minimal surface in $\mathbb{L}^{3}$, where $m\in \mathbb{R}-\left\{
-1,0,1\right\} $.
\end{theorem}

\begin{proof}
The\textbf{\ }coefficients\textbf{\ }of the\textbf{\ }first fundamental form
of the $\mathfrak{B}_{m}$ are%
\begin{eqnarray*}
E &=&4r^{2m-4}\left( \sin \theta \cos \theta \right) ^{m-1}\left(
1+r^{2}\sin \theta \cos \theta \right) ^{2}, \\
F &=&2r^{2m-3}\left( \sin \theta \cos \theta \right) ^{m-2}\left(
1+r^{2}\sin \theta \cos \theta \right) ^{2}\cos \left( 2\theta \right) , \\
G &=&-4r^{2m-2}\left( \sin \theta \cos \theta \right) ^{m-1}\left(
1+r^{2}\sin \theta \cos \theta \right) ^{2}.
\end{eqnarray*}%
Then we have%
\begin{equation*}
\det I=-\left[ 2r^{2m-3}\left( \sin \theta \cos \theta \right) ^{m-2}\left(
1+r^{2}\sin \theta \cos \theta \right) ^{2}\right] ^{2}.
\end{equation*}%
So, $\mathfrak{B}_{m}$ is a timelike surface. The Gauss map is%
\begin{equation*}
e=\frac{1}{1+r^{2}\sin \theta \cos \theta }\left( 
\begin{array}{c}
r\left( \cos \theta -\sin \theta \right) \\ 
r\left( \cos \theta +\sin \theta \right) \\ 
r^{2}\cos \theta \sin \theta -1%
\end{array}%
\right) .
\end{equation*}%
The\textbf{\ }coefficients of the\ second fundamental form of the surface are%
\begin{eqnarray*}
L &=&-2r^{m-2}(\sin ^{m}\left( \theta \right) +\cos ^{m}\left( \theta
\right) ), \\
M &=&2r^{m-1}(\sin \left( \theta \right) \cos ^{m-1}\left( \theta \right)
-\cos \left( \theta \right) \sin ^{m-1}\left( \theta \right) ), \\
N &=&-2r^{m}(\sin ^{2}\left( \theta \right) \cos ^{m-2}\left( \theta \right)
+\cos ^{2}\left( \theta \right) \sin ^{m-2}\left( \theta \right) ).
\end{eqnarray*}%
We have%
\begin{equation*}
\det II=-4r^{2m-2}\left( \sin \theta \cos \theta \right) ^{m-2}.
\end{equation*}%
Hence, the Gaussian curvature and the mean curvature, respectively, are%
\begin{equation*}
K=\left( \sin \theta \cos \theta \right) ^{2-m}\left( \frac{r^{2-m}}{\left(
1+r^{2}\sin \theta \cos \theta \right) ^{2}}\right) ^{2},
\end{equation*}%
and%
\begin{equation*}
H=0.
\end{equation*}%
So, the $\mathfrak{B}_{m}$ is a timelike minimal surface in $\mathbb{L}^{3}.$
\end{proof}

\subsection{Bour's timelike surface $\mathfrak{B}_{3}$}

If take $m=3$ in $\mathfrak{B}_{m}\left( r,\theta \right) $, we have \textit{%
Bour's timelike minimal surface}\textbf{\ }(see Fig. 14)%
\begin{equation}
\mathfrak{B}_{3}\left( r,\theta \right) =\left( 
\begin{array}{c}
\left( -\frac{r^{2}}{2}-\frac{r^{4}}{4}\right) \cos \left( 2\theta \right)
\\ 
\frac{r^{2}}{2}-\frac{r^{4}}{4}\left( \cos ^{4}\left( \theta \right) +\sin
^{4}\left( \theta \right) \right) \\ 
2\frac{r^{3}}{3}\left( \cos ^{3}\left( \theta \right) +\sin ^{3}\left(
\theta \right) \right)%
\end{array}%
\right)  \label{12}
\end{equation}%
in Minkowski 3-space, where $r\in \lbrack -1,1]$, $\theta \in \lbrack 0,\pi
].$

\begin{center}
\begin{equation*}
\FRAME{itbpF}{1.7729in}{1.8965in}{0in}{}{}{Figure}{\special{language
"Scientific Word";type "GRAPHIC";display "USEDEF";valid_file "T";width
1.7729in;height 1.8965in;depth 0in;original-width 3.1566in;original-height
3.8337in;cropleft "0";croptop "1";cropright "1";cropbottom "0";tempfilename
'MG5RC31G.wmf';tempfile-properties "XPR";}}\FRAME{itbpF}{1.8732in}{1.8983in}{%
0in}{}{}{Figure}{\special{language "Scientific Word";type "GRAPHIC";display
"USEDEF";valid_file "T";width 1.8732in;height 1.8983in;depth
0in;original-width 3.5103in;original-height 3.1877in;cropleft "0";croptop
"1";cropright "1";cropbottom "0";tempfilename
'MG5RC31H.wmf';tempfile-properties "XPR";}}
\end{equation*}

$\left( a\right) $ \ \ \ \ \ \ \ \ \ \ \ \ \ \ \ \ \ \ \ \ \ \ \ \ \ \ \ \ \
\ \ \ \ $\left( b\right) $ \ \ 

Figure 14 $\ $Bour's timelike\textbf{\ }minimal surface $\mathfrak{B}%
_{3}\left( r,\theta \right) $
\end{center}

The\textbf{\ }coefficients\textbf{\ }of the\textbf{\ }first fundamental form
of the Bour's timelike minimal surface of value $3$ are%
\begin{eqnarray*}
E &=&4r^{2}\left( \sin \theta \cos \theta \right) ^{2}\left( 1+r^{2}\sin
\theta \cos \theta \right) ^{2}, \\
F &=&r^{3}\sin 2\theta \left( 1+r^{2}\sin \theta \cos \theta \right)
^{2}\cos \left( 2\theta \right) , \\
G &=&-4r^{4}\left( \sin \theta \cos \theta \right) ^{2}\left( 1+r^{2}\sin
\theta \cos \theta \right) ^{2}.
\end{eqnarray*}%
Then%
\begin{equation*}
\det I=-4r^{6}\left( \sin \theta \cos \theta \right) ^{2}\left( 1+r^{2}\sin
\theta \cos \theta \right) ^{4}.
\end{equation*}%
The\textbf{\ }coefficients of the\ second fundamental form of the surface are%
\begin{eqnarray*}
L &=&-2r(\sin ^{3}\left( \theta \right) +\cos ^{3}\left( \theta \right) ), \\
M &=&2r^{2}(\sin \left( \theta \right) \cos ^{2}\left( \theta \right) -\cos
\left( \theta \right) \sin ^{2}\left( \theta \right) ), \\
N &=&-2r^{3}(\sin ^{2}\left( \theta \right) \cos \left( \theta \right) +\cos
^{2}\left( \theta \right) \sin \left( \theta \right) ).
\end{eqnarray*}%
So,%
\begin{equation*}
\det II=-4r^{4}\sin \theta \cos \theta .
\end{equation*}%
The mean and the Gaussian curvatures of the Bour's minimal surface of value $%
3$ are, respectively,%
\begin{equation*}
H=0,\text{ \ }K=\frac{1}{r^{2}\sin \theta \cos \theta \left( 1+r^{2}\sin
\theta \cos \theta \right) ^{4}}.
\end{equation*}

The Weierstrass patch determined by the functions%
\begin{equation*}
(\emph{f,g})=(u,u)\ \ \ \text{and \ }\ (\mathfrak{f},\mathfrak{g})=(v,v)
\end{equation*}%
in $\mathbb{L}^{3}.$

The parametric form of the surface (see Fig. 15) is%
\begin{equation}
\mathfrak{B}_{3}\left( u,v\right) =\left( 
\begin{array}{c}
-\frac{1}{2}\left( u^{2}-v^{2}\right) -\frac{1}{4}\left( u^{4}-v^{4}\right)
\\ 
\frac{1}{2}\left( u^{2}+v^{2}\right) -\frac{1}{4}\left( u^{4}+v^{4}\right)
\\ 
\frac{2}{3}\left( u^{3}+v^{3}\right)%
\end{array}%
\right) ,  \label{13}
\end{equation}%
where $u,v\in I\subset \mathbb{R}$.

\begin{center}
\begin{equation*}
\FRAME{itbpF}{1.7115in}{1.7616in}{0in}{}{}{Figure}{\special{language
"Scientific Word";type "GRAPHIC";maintain-aspect-ratio TRUE;display
"USEDEF";valid_file "T";width 1.7115in;height 1.7616in;depth
0in;original-width 2.885in;original-height 2.9689in;cropleft "0";croptop
"1";cropright "1";cropbottom "0";tempfilename
'MG5RC31I.wmf';tempfile-properties "XPR";}}\FRAME{itbpF}{1.9796in}{1.7634in}{%
0in}{}{}{Figure}{\special{language "Scientific Word";type
"GRAPHIC";maintain-aspect-ratio TRUE;display "USEDEF";valid_file "T";width
1.9796in;height 1.7634in;depth 0in;original-width 3.333in;original-height
2.9689in;cropleft "0";croptop "1";cropright "1";cropbottom "0";tempfilename
'MG5RC31J.wmf';tempfile-properties "XPR";}}
\end{equation*}

$\left( a\right) $ \ \ \ \ \ \ \ \ \ \ \ \ \ \ \ \ \ \ \ \ \ \ \ \ \ \ \ \ \
\ \ \ \ \ \ $\left( b\right) $ \ \ 

Figure 15 $\ $Timelike minimal surface $\mathfrak{B}_{3}\left( u,v\right) ,$ 
$u,v\in \lbrack -1,1]$
\end{center}

The\textbf{\ }coefficients\textbf{\ }of the\textbf{\ }first fundamental form
of the timelike Bour's surface of value $3$ are%
\begin{equation*}
E=0=G,\text{ \ }F=2uv\left( 1+uv\right) ^{2},
\end{equation*}%
So,%
\begin{equation*}
\det I=-4u^{2}v^{2}\left( 1+uv\right) ^{4}.
\end{equation*}%
The Gauss map of the surface $\mathfrak{B}_{3}$ is%
\begin{equation*}
e=\frac{1}{1+uv}\left( u-v,u+v,uv-1\right) .
\end{equation*}%
The\textbf{\ }coefficients of the\ second fundamental form of the surface are%
\begin{equation*}
L=-2u,\text{ }M=0,\text{ }N=-2v.
\end{equation*}%
Then,%
\begin{equation*}
\det II=4uv.
\end{equation*}%
The mean and the Gaussian curvatures of the timelike Bour's minimal surface $%
\mathfrak{B}_{3}$ are%
\begin{equation*}
H=0,\text{ \ }K=-\frac{1}{uv\left( 1+uv\right) ^{4}}.
\end{equation*}

\subsection{Applications of the indefinite case in $\mathbb{L}^{3}$}

\begin{example}
If take $m=2$, we have\textbf{\ }$\mathfrak{B}_{2}\left( r,\theta \right) $
(see Fig. 16)%
\begin{equation*}
\left( 
\begin{array}{c}
-r\left( \cos \left( \theta \right) -\sin \left( \theta \right) \right) -%
\frac{r^{3}}{3}\left( \cos ^{3}\left( \theta \right) -\sin ^{3}\left( \theta
\right) \right) \\ 
r\left( \cos \left( \theta \right) +\sin \left( \theta \right) \right) -%
\frac{r^{3}}{3}\left( \cos ^{3}\left( \theta \right) +\sin ^{3}\left( \theta
\right) \right) \\ 
r^{2}%
\end{array}%
\right)
\end{equation*}%
in Minkowski 3-space, where $r\in \lbrack -2,2]$, $\theta \in \lbrack -\pi
/2,\pi /2].$
\end{example}

\begin{center}
\begin{equation*}
\FRAME{itbpF}{1.8792in}{1.8325in}{0in}{}{}{Figure}{\special{language
"Scientific Word";type "GRAPHIC";maintain-aspect-ratio TRUE;display
"USEDEF";valid_file "T";width 1.8792in;height 1.8325in;depth
0in;original-width 3.3122in;original-height 3.2292in;cropleft "0";croptop
"1";cropright "1";cropbottom "0";tempfilename
'MG5RC31K.wmf';tempfile-properties "XPR";}}\FRAME{itbpF}{1.9182in}{1.8256in}{%
0in}{}{}{Figure}{\special{language "Scientific Word";type
"GRAPHIC";maintain-aspect-ratio TRUE;display "USEDEF";valid_file "T";width
1.9182in;height 1.8256in;depth 0in;original-width 2.9585in;original-height
2.8124in;cropleft "0";croptop "1";cropright "1";cropbottom "0";tempfilename
'MG5RC31L.wmf';tempfile-properties "XPR";}}
\end{equation*}

$\left( a\right) $ \ \ \ \ \ \ \ \ \ \ \ \ \ \ \ \ \ \ \ \ \ \ \ \ \ \ \ \ \
\ \ \ \ \ \ \ \ $\left( b\right) $ \ \ 

Figure 16 $\ $Bour's timelike\textbf{\ }minimal surface $\mathfrak{B}%
_{2}\left( r,\theta \right) $
\end{center}

\begin{example}
If take $m=2$, we have $\mathfrak{B}_{2}\left( r,\theta \right) $ (see Fig.
17) in Minkowski 3-space, where $r\in \lbrack -3,3]$, $\theta \in \lbrack
-\pi /2,\pi /2].$
\end{example}

\begin{center}
\begin{equation*}
\FRAME{itbpF}{2.0998in}{1.7513in}{0in}{}{}{Figure}{\special{language
"Scientific Word";type "GRAPHIC";maintain-aspect-ratio TRUE;display
"USEDEF";valid_file "T";width 2.0998in;height 1.7513in;depth
0in;original-width 3.7291in;original-height 3.1038in;cropleft "0";croptop
"1";cropright "1";cropbottom "0";tempfilename
'MG5RC31M.wmf';tempfile-properties "XPR";}}\FRAME{itbpF}{1.9311in}{1.7538in}{%
0in}{}{}{Figure}{\special{language "Scientific Word";type
"GRAPHIC";maintain-aspect-ratio TRUE;display "USEDEF";valid_file "T";width
1.9311in;height 1.7538in;depth 0in;original-width 3.0208in;original-height
2.7397in;cropleft "0";croptop "1";cropright "1";cropbottom "0";tempfilename
'MG5RC31N.wmf';tempfile-properties "XPR";}}
\end{equation*}

$\ \ \left( a\right) $ \ \ \ \ \ \ \ \ \ \ \ \ \ \ \ \ \ \ \ \ \ \ \ \ \ \ \
\ \ \ \ \ \ \ \ $\left( b\right) $

Figure 17 $\ $Bour's timelike\textbf{\ }minimal surface $\mathfrak{B}%
_{2}\left( r,\theta \right) $
\end{center}

\begin{example}
If take $m=4$, we have\textbf{\ }$\mathfrak{B}_{4}\left( r,\theta \right) $
(see Fig. 18)%
\begin{equation*}
\left( 
\begin{array}{c}
-\frac{r^{3}}{3}\left( \cos ^{3}\left( \theta \right) -\sin ^{3}\left(
\theta \right) \right) -\frac{r^{5}}{5}\left( \cos ^{5}\left( \theta \right)
-\sin ^{5}\left( \theta \right) \right) \\ 
\frac{r^{3}}{3}\left( \cos ^{3}\left( \theta \right) +\sin ^{3}\left( \theta
\right) \right) -\frac{r^{5}}{5}\left( \cos ^{5}\left( \theta \right) +\sin
^{5}\left( \theta \right) \right) \\ 
\frac{r^{4}}{4}\left( \cos ^{4}\left( \theta \right) +\sin ^{4}\left( \theta
\right) \right)%
\end{array}%
\right)
\end{equation*}%
in Minkowski 3-space, where $r\in \lbrack -1,1],$ $\theta \in \lbrack 0,\pi
].$
\end{example}

\begin{center}
\begin{equation*}
\FRAME{itbpF}{2.0583in}{1.8983in}{0in}{}{}{Figure}{\special{language
"Scientific Word";type "GRAPHIC";maintain-aspect-ratio TRUE;display
"USEDEF";valid_file "T";width 2.0583in;height 1.8983in;depth
0in;original-width 3.3226in;original-height 3.0623in;cropleft "0";croptop
"1";cropright "1";cropbottom "0";tempfilename
'MG5RC31O.wmf';tempfile-properties "XPR";}}\FRAME{itbpF}{1.9778in}{1.9069in}{%
0in}{}{}{Figure}{\special{language "Scientific Word";type
"GRAPHIC";maintain-aspect-ratio TRUE;display "USEDEF";valid_file "T";width
1.9778in;height 1.9069in;depth 0in;original-width 2.9066in;original-height
2.802in;cropleft "0";croptop "1";cropright "1";cropbottom "0";tempfilename
'MG5RC31P.wmf';tempfile-properties "XPR";}}
\end{equation*}

$\left( a\right) $ \ \ \ \ \ \ \ \ \ \ \ \ \ \ \ \ \ \ \ \ \ \ \ \ \ \ \ \ \
\ \ \ \ \ \ \ \ \ $\left( b\right) $

Figure 18 $\ $Bour's timelike\textbf{\ }minimal surface $\mathfrak{B}%
_{4}\left( r,\theta \right) $
\end{center}

\begin{example}
If take $m=4$, we have $\mathfrak{B}_{2}\left( r,\theta \right) $ (see Fig.
19) in Minkowski 3-space, where $r\in \lbrack -2,2]$, $\theta \in \lbrack
0,\pi /2].$
\end{example}

\begin{center}
\begin{equation*}
\FRAME{itbpF}{2.0453in}{1.6873in}{0in}{}{}{Figure}{\special{language
"Scientific Word";type "GRAPHIC";maintain-aspect-ratio TRUE;display
"USEDEF";valid_file "T";width 2.0453in;height 1.6873in;depth
0in;original-width 3.5518in;original-height 2.9274in;cropleft "0";croptop
"1";cropright "1";cropbottom "0";tempfilename
'MG5RC31Q.wmf';tempfile-properties "XPR";}}\FRAME{itbpF}{1.8498in}{1.7045in}{%
0in}{}{}{Figure}{\special{language "Scientific Word";type
"GRAPHIC";maintain-aspect-ratio TRUE;display "USEDEF";valid_file "T";width
1.8498in;height 1.7045in;depth 0in;original-width 3.0208in;original-height
2.7812in;cropleft "0";croptop "1";cropright "1";cropbottom "0";tempfilename
'MG5RC31R.wmf';tempfile-properties "XPR";}}
\end{equation*}

$\left( a\right) $ \ \ \ \ \ \ \ \ \ \ \ \ \ \ \ \ \ \ \ \ \ \ \ \ \ \ \ \ \
\ \ \ \ \ \ \ $\left( b\right) $

Figure 19 $\ $Bour's timelike\textbf{\ }minimal surface $\mathfrak{B}%
_{4}\left( r,\theta \right) $
\end{center}

\begin{example}
If take $m=5$, we have\textbf{\ }$\mathfrak{B}_{5}\left( r,\theta \right) $
(see Fig. 20)%
\begin{equation*}
\left( 
\begin{array}{c}
-\frac{r^{4}}{4}\left( \cos ^{4}\left( \theta \right) -\sin ^{4}\left(
\theta \right) \right) -\frac{r^{6}}{6}\left( \cos ^{6}\left( \theta \right)
-\sin ^{6}\left( \theta \right) \right) \\ 
\frac{r^{4}}{4}\left( \cos ^{4}\left( \theta \right) +\sin ^{4}\left( \theta
\right) \right) -\frac{r^{6}}{6}\left( \cos ^{6}\left( \theta \right) +\sin
^{6}\left( \theta \right) \right) \\ 
\frac{r^{5}}{5}\left( \cos ^{5}\left( \theta \right) +\sin ^{5}\left( \theta
\right) \right)%
\end{array}%
\right)
\end{equation*}%
in Minkowski 3-space, where $r\in \lbrack -0.003,0.003]$, $\theta \in
\lbrack 0,\pi ].$
\end{example}

\begin{center}
\begin{equation*}
\FRAME{itbpF}{1.8403in}{1.8723in}{0in}{}{}{Figure}{\special{language
"Scientific Word";type "GRAPHIC";maintain-aspect-ratio TRUE;display
"USEDEF";valid_file "T";width 1.8403in;height 1.8723in;depth
0in;original-width 3.5552in;original-height 3.6175in;cropleft "0";croptop
"1";cropright "1";cropbottom "0";tempfilename
'MG5RC31S.bmp';tempfile-properties "XPR";}}\FRAME{itbpF}{1.9588in}{1.8905in}{%
0in}{}{}{Figure}{\special{language "Scientific Word";type
"GRAPHIC";maintain-aspect-ratio TRUE;display "USEDEF";valid_file "T";width
1.9588in;height 1.8905in;depth 0in;original-width 2.9482in;original-height
2.8435in;cropleft "0";croptop "1";cropright "1";cropbottom "0";tempfilename
'MG5RC31T.wmf';tempfile-properties "XPR";}}
\end{equation*}

$\left( a\right) $ \ \ \ \ \ \ \ \ \ \ \ \ \ \ \ \ \ \ \ \ \ \ \ \ \ \ \ \ \
\ \ \ \ \ $\left( b\right) $

Figure 20 $\ $Bour's timelike\textbf{\ }minimal surface $\mathfrak{B}%
_{5}\left( r,\theta \right) $
\end{center}

\textbf{Acknowledgement.} A large part of this work had been completed by
the author, when he visited as a post-doctoral researcher at the Katholieke
Universiteit Leuven, Belgium in 2011-2012 academic year. The author would
like to thanks to the hospitality of the members of the geometry section of
the K.U. Leuven, especially to the Professor Franki Dillen.

\end{document}